\documentclass[11pt,leqno]{article}
\usepackage{amsmath, amscd, amsthm, amssymb, graphics, xypic, mathrsfs,setspace,fancyhdr}
\usepackage[pagebackref=true]{hyperref}
\hypersetup{backref}

\newcommand{\tensor}{\otimes}
\newcommand{\colim}{\operatorname{colim}}
\newcommand{\Spec}{\operatorname{Spec}}

\newcommand{\isomt}{{\stackrel{{\scriptscriptstyle{\sim}}}{\;\rightarrow\;}}}

\newcommand{\bpi}{\boldsymbol{\pi}}

\renewcommand{\hom}{\operatorname{Hom}}

\newcommand{\Q}{{\mathbb Q}}
\newcommand{\Z}{{\mathbb Z}}
\newcommand{\aone}{{\mathbb A}^1}
\newcommand{\pone}{{\mathbb P}^1}

\newcommand{\gm}{{{\mathbb G}_{m}}}

\newcommand{\ho}[1]{{\mathcal H}({#1})}

\newcommand{\Daone}{{\operatorname{D}_{\aone}}}

\newcommand{\SH}{{\mathbf{SH}}}

\newcommand{\Sm}{{\mathcal Sm}}

\newcommand{\Spc}{{\mathcal Spc}}

\newcommand{\Ab}{{\mathcal Ab}}

\renewcommand{\H}{{{\mathbf H}}}

\newcommand{\F}{{\mathcal F}}

\newcounter{intro}
\theoremstyle{plain}
\newtheorem{thm}{Theorem}[section]

\newtheorem{lem}[thm]{Lemma}
\newtheorem{cor}[thm]{Corollary}
\newtheorem{prop}[thm]{Proposition}

\newtheorem*{thm*}{Theorem}
\newtheorem*{problem*}{Problem}

\newtheorem{thmintro}{Theorem}

\theoremstyle{definition}
\newtheorem{defn}[thm]{Definition}

\theoremstyle{remark}
\newtheorem{rem}[thm]{Remark}
\newtheorem{remintro}[thmintro]{Remark}

\newtheorem{ex}[thm]{Example}

\numberwithin{equation}{section}

\begin{document}
\pagestyle{fancy}
\renewcommand{\sectionmark}[1]{\markright{\thesection\ #1}}
\fancyhead{}
\fancyhead[LO,R]{\bfseries\footnotesize\thepage}
\fancyhead[LE]{\bfseries\footnotesize\rightmark}
\fancyhead[RO]{\bfseries\footnotesize\rightmark}
\chead[]{}
\cfoot[]{}
\setlength{\headheight}{1cm}

\author{\begin{small}Aravind Asok\thanks{Aravind Asok was partially supported by National Science Foundation Awards DMS-0900813 and DMS-0966589.}\end{small} \\ \begin{footnotesize}Department of Mathematics\end{footnotesize} \\ \begin{footnotesize}University of Southern California\end{footnotesize} \\ \begin{footnotesize}Los Angeles, CA 90089-2532 \end{footnotesize} \\ \begin{footnotesize}\url{asok@usc.edu}\end{footnotesize} \and \begin{small}Christian Haesemeyer\thanks{Christian Haesemeyer was partially supported by National Science Foundation Award DMS-0966821.}\end{small} \\ \begin{footnotesize}Department of Mathematics\end{footnotesize} \\ \begin{footnotesize}University of California, Los Angeles\end{footnotesize} \\ \begin{footnotesize}Los Angeles, CA 90095-1555 \end{footnotesize} \\ \begin{footnotesize}\url{chh@math.ucla.edu}\end{footnotesize}}

\title{{\bf Stable $\aone$-homotopy and $R$-equivalence}}
\date{}
\maketitle

\begin{abstract}
We prove that existence of a $k$-rational point can be detected by the stable $\aone$-homotopy category of $S^1$-spectra, or even a ``rationalized" variant of this category.
\end{abstract}


\section{Introduction}
Suppose $k$ is a field and $X$ a smooth proper $k$-variety.  By the Lang-Nishimura lemma \cite{Nishimura}, one knows that existence of a $k$-point is a $k$-birational invariant.  By a remark of Morel and Voevodsky, one also knows that existence of a $k$-rational point is an unstable $\aone$-homotopy invariant; see, e.g., \cite[\S 3 Remark 2.5]{MV}, where it is observed that this is a consequence of the fact that the Nisnevich topology is used in the construction of the unstable $\aone$-homotopy category.  The purpose of this note is to, in a sense, combine the two results above and to show that the ability to detect rational points persists in the Morel-Voevodsky stable $\aone$-homotopy category of $S^1$-spectra as well as in Morel's $\aone$-derived category and even the $\aone$-derived category with $\Q$-coefficients.  Very loosely speaking, our results say that existence of rational points can be detected by purely cohomological means.

Write $\SH^{s}_k$ for the Morel-Voevodsky stable $\aone$-homotopy category of $S^1$-spectra (see \cite[Definition 4.1.1]{MStable} for a precise definition).  Let $\Sigma^{\infty}_s X_+$ denote the $\aone$-localization of the simplicial suspension spectrum of $X$ with a disjoint basepoint attached.  The $0$-th $S^1$-stable $\aone$-homotopy sheaf of $X$, denoted $\bpi_0^{s}(X_+)$, is the Nisnevich sheaf associated with the presheaf on $\Sm_k$
\[
U \longmapsto \hom_{\SH^{s}_k}(\Sigma^{\infty}_s U_+,\Sigma^{\infty}_s X_+).
\]
The structure morphism $X \to \Spec k$ induces a morphism of sheaves $\bpi_0^s(X_+) \to \bpi_0^s{\Spec k_+}$.  The sheaf $\bpi_0^{s}(X_+)$ is a birational invariant of smooth, proper $k$-varieties.  If $X$ has a $k$-rational point, the map $\bpi_0^s(X_+) \to \bpi_0^{s}(\Spec k_+)$ is a split epimorphism.  We prove a converse to this statement.

\begin{thmintro}
\label{thmintro:main}
If $X$ is a smooth proper $k$-variety, then the following conditions are equivalent:
\begin{itemize}
\item[i)] $X$ has a $k$-rational point,
\item[ii)] there is a stable $\aone$-homotopy class of maps $\Sigma^{\infty}_s \Spec k_+ \to \Sigma^{\infty}_s X_+$ splitting the structure map $\Sigma^{\infty}_s X_+ \to \Sigma^{\infty}_s \Spec k_+$, and
\item[iii)] the morphism of sheaves $\bpi_0^s(X_+) \to \bpi_0^{s}(\Spec k_+)$ is a split epimorphism.
\end{itemize}
\end{thmintro}

That (i) $\Longrightarrow$ (ii) $\Longrightarrow$ (iii) is clear, and the work goes into showing (iii) $\Longrightarrow$ (i).  Proposition \ref{prop:hurewicz} allows us to show that $\bpi_0^s(\Spec k)_+ = \Z$, so (iii) can be made more explicit.  For a corresponding statement with $\Q$-coefficients see Remark \ref{rem:rationalizedversion}.  To put this result in context, we observe how these results combined with those of \cite{AH} give a framework for comparing rational points and $0$-cycles of degree $1$.

\begin{remintro}
Let $\Sigma_{\pone}$ denote the operation of smashing with the simplicial suspension spectrum of $(\pone,\infty)$, and let $\Omega_{\pone}$ be the adjoint looping functor.  If $E$ is any $S^1$-spectrum, there is a map $E \to \Omega_{\pone}\Sigma_{\pone} E$.  We can iterate this functor to obtain a tower
\[
E \longrightarrow \Omega_{\pone}\Sigma_{\pone} E \longrightarrow \Omega_{\pone}^2\Sigma_{\pone}^2 E \cdots.
\]
The $0$-th $\pone$-stable $\aone$-homotopy sheaf of $E$, denoted $\bpi_0^{s\aone}(E)$, can be computed by means of the formula
\[
\bpi_0^{s\aone}(E) = \colim_n \bpi_0^{s}(\Omega_{\pone}^n\Sigma_{\pone}^n E).
\]
The structure map $X \to \Spec k$ induces a morphism $\bpi_0^{s\aone}(X_+) \to \bpi_0^{s\aone}(\Spec k_+)$.  One says that $X$ has a rational point up to stable $\aone$-homotopy if the latter map is a split epimorphism.  By \cite[Theorem 1]{AH}, if $k$ is an infinite perfect field having characteristic unequal to $2$, we know that a smooth proper $k$-scheme $X$ has a $0$-cycle of degree $1$ if and only if it has a rational point up to stable $\aone$-homotopy.  Thus, under the stated hypotheses on $k$, the difference between a $0$-cycle of degree $1$ and $k$-rational point is measured by the difference between $S^1$-stable and $\pone$-stable $\aone$-homotopy theory.  The existence of such a connection between rational points and $0$-cycles of degree $1$ was suggested in \cite[p. 395-6]{LevineSlices}.
\end{remintro}

\subsubsection*{Acknowledgements}
This work was originally begun in the context of a project between the authors and Fabien Morel; we thank him for his collaboration in the early stages, especially his insistence on studying the sheaf $\Z(\bpi_0^{b\aone}(X))$ discussed below.  We also thank Marc Levine for interesting discussions.


\section{Proof of Theorem \ref{thmintro:main}}
Let us introduce some notation for the rest of the note.  Throughout, suppose $k$ is a field.  Let $\Sm_k$ denote the category of schemes that are separated, smooth, and have finite type over $\Spec k$.  Write $\Spc_k$ for the category of simplicial Nisnevich sheaves of sets on $\Sm_k$; objects of this category will be called spaces.  We identify Nisnevich sheaves with the corresponding simplicial objects.

Write $\ho{k}$ for the Morel-Voevodsky unstable $\aone$-homotopy category.  This category is constructed from the category $\Spc_k$ by localizing at the class of $\aone$-weak equivalences (see \cite[\S 3.2]{MV}).  We write $\SH^{s}_k$ for the stable $\aone$-homotopy category of $S^1$-spectra, e.g., as defined in \cite[\S 5]{MStable}.  Loosely speaking, this category is obtained from $\ho{k}$ by formally inverting the simplicial suspension operation.  We write $\SH_k$ for the stable $\aone$-homotopy category of $\pone$-spectra, e.g., as defined in \cite{JardineSymmetric}; this category is obtained from $\SH^s_k$ by further inverting the operation of smashing with the suspension spectrum of $\gm$.

Recall that a presheaf of sets $\F$ on $\Sm_k$ is called {\em $\aone$-invariant}, if for any smooth scheme $U$ the map $\F(U) \to \F(U \times \aone)$ induced by pullback along the projection $U \times \aone \to U$ is a bijection.  If $\mathcal{X}$ is any space, we write $\Z(\mathcal{X})$ for the simplicial sheaf of abelian groups freely generated by the simplices of $\mathcal{X}$.  The normalized chain complex of $\Z({\mathcal X})$, for which we will write $C_*(\Z({\mathcal X}))$, is a chain complex of sheaves of abelian groups.

Write $D(\Ab_{Nis}(k))$ for the (unbounded) derived category of Nisnevich sheaves of abelian groups on $\Sm_k$.  A complex of sheaves of abelian groups $A$ on $\Sm_k$ is called {\em $\aone$-local} if for any complex $B$ the map
\[
\hom_{D(\Ab_{Nis}(k))}(B,A) \longrightarrow \hom_{D(\Ab_{Nis}(k))}(B \tensor \Z(\aone),A)
\]
is a bijection.  A sheaf $\F$ of abelian groups is said to be {\em strictly $\aone$-invariant} if it is $\aone$-invariant viewed as a complex of sheaves situated in degree $0$.  Consider the full subcategory of $D(\Ab_{Nis}(k))$ consisting of $\aone$-local complexes; the inclusion of this subcategory into $D(\Ab_{Nis}(k))$ admits a left adjoint $L_{\aone}$ called the functor of $\aone$-localization \cite[Proposition 4.3]{CisinskiDeglise}.  Morel's $\aone$-derived category $\Daone(k)$ is (equivalent to) the full subcategory of the derived category of Nisnevich sheaves of abelian groups consisting of $\aone$-local complexes.

Set $C_*^{\aone}(\mathcal{X}) := L_{\aone}C_*(\Z({\mathcal X}))$; this complex is called the $\aone$-chain complex of $\mathcal{X}$.   The $0$-th $\aone$-homology sheaf of $\mathcal{X}$, denoted $\H_0^{\aone}({\mathcal{X}})$, is just the $0$-th homology sheaf of $C_*^{\aone}(\mathcal{X})$.  The functor $\mathcal{X} \mapsto C_*^{\aone}(\mathcal{X})$ induces a functor $\ho{k} \to \Daone(k)$.  The suspension isomorphism for homology shows that this functor factors through a functor $\SH^s_k \to \Daone(k)$ that we will call abelianization.  For recollections about the $\aone$-derived category, see \cite[\S 3.2]{MField}.

\subsubsection*{The Hurewicz homomorphism}
The abelianization functor induces a Hurewicz morphism $\bpi_0^{s}(\mathcal{X}_+) \to \H_0^{\aone}(\mathcal{X})$ (note: the definition of $\bpi_0^s(-)$ given in the introduction makes sense for any $\mathcal{X} \in \Spc_k$).  The following result is a consequence of the stable $\aone$-connectivity theorem \cite[Theorem 6.1.8]{MStable}, which states that $(-1)$-connected spectra or complexes are preserved by $\aone$-localization.

\begin{prop}
\label{prop:hurewicz}
If $\mathcal{X}$ is a space, the canonical morphism $\bpi_0^{s}(\mathcal{X}_+) \to \H_0^{\aone}(\mathcal{X})$ is an isomorphism of strictly $\aone$-invariant sheaves.
\end{prop}

Because of this proposition, we can (and will) replace the $0$-th stable $\aone$-homotopy sheaf by the $0$-th $\aone$-homology sheaf of a space in the sequel.  The next result follows immediately from Proposition \ref{prop:hurewicz} and, e.g., \cite[Theorem 2.2.9]{ABirational}.

\begin{cor}
If $k$ is an infinite field, the sheaf $\bpi_0^{s}(X_+)$ is a birational invariant of smooth and proper $k$-varieties.
\end{cor}

\subsection{Strict $\aone$-invariance and birationality}
\begin{defn}
Suppose $\F$ is a presheaf of sets on $\Sm_k$.  We say $\F$ is {\em birational} if for any open dense immersion $U \to U'$ in $\Sm_k$, the map $\F(U') \to \F(U)$ is an isomorphism.
\end{defn}

In the following lemma, we summarize some technical properties of birational presheaves.  This result is ``well known to the experts" and we include it for the convenience of the reader; results along these lines can also be found in, e.g., \cite[\S 2]{LevineSlices}.

\begin{lem}
\label{lem:birationalaoneinvariantimpliesstrictlyaoneinvariant}
If $\F$ is a birational presheaf, the free presheaf of abelian groups $\Z(\F)$ is also birational, and both $\F$ and $\Z(\F)$ are Nisnevich sheaves.  If $\F$ is furthermore $\aone$-invariant, then
$\Z(\F)$ is Nisnevich flasque, and $\Z(\F)$ is strictly $\aone$-invariant.
\end{lem}

\begin{proof}
To show that $\F$ is a Nisnevich sheaf is, we just have to check that $\F$ takes an elementary distinguished square
\[
\xymatrix{
V' \ar[r]\ar[d] & V \ar[d]^{\psi} \\
U \ar[r] & X,
}
\]
(where $\psi$ is \'etale, $U \to X$ is an open immersion, $X \setminus U$ is given the usual reduced scheme structure, and the map $\psi^{-1}(X \setminus U) \to X \setminus U$ is an isomorphism) to a cartesian square.  Since $\F$ is birational, both the bottom and top maps are isomorphisms and so the diagram is cartesian.  Now, if $\F$ is birational, then by definition $\Z(\F)$ is also birational, and by what we just showed $\Z(\F)$ is also a Nisnevich sheaf.

If $\F$ is also $\aone$-invariant, it follows immediately that $\Z(\F)$ is also $\aone$-invariant.  We will now show that $\Z(\F)$ is Nisnevich flasque.  To see this, recall that the Nisnevich cohomology can be computed by means of \u Cech cochains: \cite[p. 95]{MV} mentions this without proof, but the proof is essentially identical to the corresponding statement in the \'etale topology; one uses \u Cech-derived functor spectral sequence and the fact \cite[Lemma 1.18.1]{Nisnevich} that the higher cohomology sheaves of a Nisnevich sheaf of abelian groups vanish.  Therefore, suppose $X$ is an irreducible smooth scheme, $u: U \to X$ is a Nisnevich cover of $X$.  By lifting the generic point $\eta$ of $X$, we can find a component of $U$ that is birational to $X$.  Since each map $U^{\times n+1} \to U^{\times n}$ is also a Nisnevich cover, it follows that $\Z(\F)(U^{\times n}) \to \Z(\F)(U^{\times n+1})$ is injective and thus all higher Nisnevich cohomology of $\Z(\F)$ vanishes.
\end{proof}

\begin{cor}
\label{cor:birationalh0}
If $\F$ is a birational and $\aone$-invariant sheaf of sets, the canonical map $\F \to \Z(\F)$ induces an isomorphism $\H_0^{\aone}(\F) \to \Z(\F)$.
\end{cor}

\begin{proof}
By definition $\H_0^{\aone}(\F) = H_0(L_{\aone}\Z(\F))$.  However, since $\Z(\F)$ is Nisnevich flasque, it follows that $\Z(\F)$ is $\aone$-local, i.e., the canonical map $L_{\aone}(\Z(\F)) \to \Z(\F)$ is an isomorphism.
\end{proof}

\begin{ex}
Suppose $X$ is an $\aone$-rigid smooth proper $k$-scheme (see \cite[\S 3 Example 2.4]{MV}).  Given an open dense immersion $U \to U'$, the map $X(U') \to X(U)$ is an isomorphism; indeed any such map is uniquely determined by where it sends the generic point of each component.  As a consequence $\Z(X)$ is a strictly $\aone$-invariant sheaf.  Because $\Z(X)$ is $\aone$-local, we see that the canonical map $C_*^{\aone}(X) = L_{\aone}\Z(X) \to \Z(X)$ is an isomorphism, and thus that $\H_0^{\aone}(X) = \Z(X)$.  Thus, $X$, $\Z(X)$, and $\H_0^{\aone}(X)$ are all birational sheaves and by Lemma \ref{lem:birationalaoneinvariantimpliesstrictlyaoneinvariant} all these sheaves are all strictly $\aone$-invariant.  As a consequence of Corollary \ref{cor:birationalh0}, we deduce that if $k$ is infinite and $X'$ is any smooth proper variety that is stably $k$-birationally equivalent to a smooth proper $\aone$-rigid variety $X$, then $\H_0^{\aone}(X') = \H_0^{\aone}(X)$.
\end{ex}

\subsection{Birational connected components and the main result}
Suppose $X$ is a smooth proper variety over a field $k$.  If $L/k$ is a separable, finitely generated extension, recall that two $L$-points in $X$ are $R$-equivalent if they can be connected by the images of a chain of morphisms from $\pone_L$ to $X$ (over $k$) \cite{Manin}.  There is a birational sheaf related to $R$-equivalence classes of points in $X$.

\begin{thm}
\label{thm:birationalpi0}
If $X$ is a smooth proper $k$-variety, there is a birational and $\aone$-invariant sheaf $\bpi_0^{b\aone}(X)$ together with a canonical map $X \to \pi_0^{b\aone}(X)$ functorial for morphisms of proper varieties such that for any separable finitely generated extension $L/k$ the induced map $X(L) \to \pi_0^{b\aone}(X)(L)$ factors through a bijection $X(L)/R \to \pi_0^{b\aone}(X)(L)$.
\end{thm}

\begin{proof}
Everything except the statement of functoriality is included in {\cite[Theorem 6.2.1]{AM}}.  Since $\bpi_0^{b\aone}(X)$ is a birational and $\aone$-invariant sheaf, to construct a morphism $\bpi_0^{b\aone}(Y) \to \bpi_0^{b\aone}(X)$, it suffices to observe that by the definition of $R$-equivalence a morphism $f: X \to Y$ induces morphisms $X(L)/R \to Y(L)/R$ for every finitely generated separable extension $L/k$.
\end{proof}

If $X$ is a smooth proper variety, we can consider the sheaf $\Z(\bpi_0^{b\aone}(X))$.  By Lemma \ref{lem:birationalaoneinvariantimpliesstrictlyaoneinvariant}, it follows that $\Z(\bpi_0^{b\aone}(X))$ is a strictly $\aone$-invariant sheaf, and Corollary \ref{cor:birationalh0} gives rise to a canonical identification $\H_0^{\aone}(\bpi_0^{b\aone}(X)) \isomt \Z(\bpi_0^{b\aone}(X))$.  As a consequence of Theorem \ref{thm:birationalpi0} we deduce the existence of a canonical morphism
\[
\phi_X: \H_0^{\aone}(X) \longrightarrow \Z(\bpi_0^{b\aone}(X)).
\]
Because $\Z(\bpi_0^{b\aone}(X))$ is a strictly $\aone$-invariant sheaf, existence of this morphism also follows immediately from \cite[Lemma 2.2.3]{ABirational}, which states that $\H_0^{\aone}(X)$ is initial among strictly $\aone$-invariant sheaves $M$ admitting a morphism of sheaves $X \to M$.


\begin{rem}
It seems reasonable to expect that the morphism $\phi_X: \H_0^{\aone}(X) \to \Z(\bpi_0^{b\aone}(X))$ is an isomorphism.  Since our goal is to get as quickly as possible to the connection with rational points we did not pursue this further.
\end{rem}

\begin{cor}
\label{cor:main}
If $X$ is a smooth proper $k$-variety, then the set $X(k)$ is non-empty if and only if the map $\H_0^{\aone}(X) \to \Z$ induced by the structure map is a split surjection.
\end{cor}

\begin{proof}
If $X(k)$ is non-empty, then we get a morphism $\Z = \H_0^{\aone}(\Spec k) \to \H_0^{\aone}(X)$ that splits the map induced by the structure morphism.  Conversely, note that the map $\H_0^{\aone}(X) \to \Z(\bpi_0^{b\aone}(X))$ is functorial in $X$, and thus the morphism $\H_0^{\aone}(X) \to \Z$ factors through the morphism $\phi_X$.  A splitting $\Z \to \H_0^{\aone}(X)$ therefore gives rise to a non-trivial morphism $\Z \to \Z(\bpi_0^{b\aone}(X))$, i.e., an element of $\Z(\bpi_0^{b\aone}(X))(k)$.  The group $\Z(\bpi_0^{b\aone}(X))(k)$ is by Theorem \ref{thm:birationalpi0} the free abelian group on the set $X(k)/R$.  Since the group $\Z(\bpi_0^{b\aone}(X))(k)$ is a non-trivial free abelian group, we deduce that $X(k)/R$ has at least $1$ element, and therefore $X(k)$ is non-empty.
\end{proof}

\begin{proof}[Proof of Theorem \ref{thmintro:main}]
Combine Corollary \ref{cor:main} and Proposition \ref{prop:hurewicz}.
\end{proof}

\begin{rem}
\label{rem:rationalizedversion}
The rationalized $\aone$-derived category is obtained by following the construction of the $\aone$-derived category sketched above and replacing abelian groups by $\Q$-vector spaces throughout.  Replacing $\Z$ by $\Q$ in all of the above allows one to deduce that existence of $k$-rational point is detected by the {\em rationalized} $\aone$-derived category.  To be precise, if $X$ is a smooth proper $k$-scheme, then $X$ has a $k$-rational point if and only if the canonical map $\H_0^{\aone}(X,\Q) \to \Q$ induced by the structure morphism $X \to \Spec k$ is a split epimorphism.  This statement also implies a statement about an appropriate ``rational" version of the stable $\aone$-homotopy category of $S^1$-spectra, but we leave this to the reader.
\end{rem}





\begin{footnotesize}
\bibliographystyle{alpha}
\bibliography{Requivalenceandaonehomology}

\begin{thebibliography}{Man86}

\bibitem[AH10]{AH}
A.~Asok and C.~Haesemeyer.
\newblock Stable {${\mathbb A}^1$}-homotopy and quadratic $0$-cycles.
\newblock 2010.
\newblock {\em In preparation}.

\bibitem[AM09]{AM}
A.~Asok and F.~Morel.
\newblock Smooth varieties up to {${\mathbb A}^1$}-homotopy and algebraic
  $h$-cobordisms.
\newblock 2009.
\newblock {\em Preprint} available at \url{http://arxiv.org/abs/0810.0324}.

\bibitem[Aso10]{ABirational}
A.~Asok.
\newblock Birational invariants and {${\mathbb A}^1$}-connectedness.
\newblock 2010.
\newblock {\em Preprint} available at \url{http://arxiv.org/abs/1001.4574}.

\bibitem[CD09]{CisinskiDeglise}
D.-C. Cisinski and F.~D{\'e}glise.
\newblock Local and stable homological algebra in {G}rothendieck abelian
  categories.
\newblock {\em Homology, Homotopy Appl.}, 11(1):219--260, 2009.

\bibitem[Jar00]{JardineSymmetric}
J.~F. Jardine.
\newblock Motivic symmetric spectra.
\newblock {\em Doc. Math.}, 5:445--553 (electronic), 2000.

\bibitem[Lev10]{LevineSlices}
M.~Levine.
\newblock Slices and transfers.
\newblock {\em Doc. Math.}, pages 393--443, 2010.
\newblock Extra Volume: Andrei A. Suslin's Sixtieth Birthday.

\bibitem[Man86]{Manin}
Yu.~I. Manin.
\newblock {\em Cubic forms}, volume~4 of {\em North-Holland Mathematical
  Library}.
\newblock North-Holland Publishing Co., Amsterdam, second edition, 1986.
\newblock Algebra, geometry, arithmetic, Translated from the Russian by M.
  Hazewinkel.

\bibitem[Mor05]{MStable}
F.~Morel.
\newblock The stable {${\mathbb A}^1$}-connectivity theorems.
\newblock {\em $K$-Theory}, 35(1-2):1--68, 2005.

\bibitem[Mor06]{MField}
F.~Morel.
\newblock {${\mathbb A}^1$}-algebraic topology over a field.
\newblock 2006.
\newblock {\em Preprint}, available at
  \url{http://www.mathematik.uni-muenchen.de/~morel/preprint.html}.

\bibitem[MV99]{MV}
F.~Morel and V.~Voevodsky.
\newblock {${\mathbb A}^1$}-homotopy theory of schemes.
\newblock {\em Inst. Hautes \'Etudes Sci. Publ. Math.}, (90):45--143 (2001),
  1999.

\bibitem[Nis55]{Nishimura}
H.~Nishimura.
\newblock Some remark[s] on rational points.
\newblock {\em Mem. Coll. Sci. Univ. Kyoto. Ser. A. Math.}, 29:189--192, 1955.

\bibitem[Nis89]{Nisnevich}
Ye.~A. Nisnevich.
\newblock The completely decomposed topology on schemes and associated descent
  spectral sequences in algebraic {$K$}-theory.
\newblock In {\em Algebraic {$K$}-theory: connections with geometry and
  topology ({L}ake {L}ouise, {AB}, 1987)}, volume 279 of {\em NATO Adv. Sci.
  Inst. Ser. C Math. Phys. Sci.}, pages 241--342. Kluwer Acad. Publ.,
  Dordrecht, 1989.

\end{thebibliography}
\end{footnotesize}
\end{document}